\documentclass[11pt]{amsart}
\usepackage{amssymb}
\usepackage{xcolor}
\usepackage{mathrsfs}
\usepackage{amsfonts}
\usepackage{amsmath}
\usepackage{hyperref}
\usepackage{mathtools}  \usepackage{amsrefs}
\mathtoolsset{showonlyrefs,showmanualtags}

\allowdisplaybreaks
\swapnumbers

\numberwithin{equation}{section}
\allowdisplaybreaks[4]
\newtheorem{Theorem}{Theorem}[section]
\newtheorem{theorem}[equation]{Theorem}

\newtheorem{Lemma}[equation]{Lemma}
\newtheorem{lemma}[equation]{Lemma}

\newtheorem{Remark}[equation]{Remark}

\newtheorem{proposition}[equation]{Proposition}

\theoremstyle{definition}

\newtheorem{definition}[equation]{Definition}

\def\dist{{\rm dist}\ }
\def\good{{\rm good}\ }

\def\bbZ{\mathbb{Z}}

\def\bbR{\mathbb{R}}

\begin{document}

\title[A Two weight theorem for square functions]{Two weight norm inequalities for the $ g$ function}


\author{Michael T Lacey}   
\address{ School of Mathematics, Georgia Institute of Technology, Atlanta GA 30332, USA}
\email {lacey@math.gatech.edu}
\thanks{Research supported in part by grant NSF-DMS 0968499 and 1265570, a grant from the Simons Foundation (\#229596 to Michael Lacey),
and the Australian Research Council through grant ARC-DP120100399.}

\author{Kangwei Li}
\address{School of Mathematical Sciences and LPMC,  Nankai University,
      Tianjin~300071, China}
\email{likangwei9@mail.nankai.edu.cn}

\begin{abstract}
Given two weights  $ \sigma , w$ on $ \mathbb R ^{n}$, the classical $ g$-function satisfies the norm inequality
$ \lVert g (f\sigma)\rVert_{L ^2 (w)} \lesssim \lVert f\rVert_{L ^2 (\sigma )}$ if and only if the two weight Muckenhoupt
$ A_2$ condition holds, and a family of testing conditions holds, namely
\begin{equation*}
\iint _{Q (I)} (\nabla P_t (\sigma \mathbf 1_I)(x, t))^2 \; dw \, t dt \lesssim \sigma (I)
\end{equation*}
uniformly over  all cubes $ I \subset \mathbb R ^{n}$, and $ Q (I)$ is the Carleson box over $ I$.
A corresponding characterization for the intrinsic square function of Wilson also holds.
\end{abstract}

\maketitle
\section{Introduction and Main Results}
By a \emph{weight} we mean a positive Borel locally finite measure on $\bbR^n$. We consider the two weight norm inequality for the classical $g$-function for a pair of weights $(w,\sigma)$ on $\bbR^n$:
\begin{equation}\label{e:g}
  \|g(f\sigma)\|_{L^2(w)}\le \mathcal G\|f\|_{L^2(\sigma)}.
\end{equation}
Here and throughout, $ \mathcal G$ denotes the best constant in the inequality.
Recall that the classical $g$-function is defined by
\begin{gather*}
  g(f)(x)=\left(\int_0^\infty |\nabla P_t f(x,t)|^2 t dt\right)^{1/2},
\\
    |\nabla P_t f (x,t)|^2:=|\frac{\partial }{\partial t} P_t f |^2+|\nabla_x P_t f|^2
\\
\textup{and} \quad   P_t f(x)= \frac{\Gamma(\frac{n+1}{2})}{\pi^{(n+1)/2}}\int _{\mathbb R ^{n}} f (x-y) \frac{t^{-n}}{(1+|y/t|^2)^{(n+1)/2}} \, dy.
   \end{gather*}
The main result is a characterization of the two weight inequality, as follows.

\begin{theorem}\label{t:g} For two weights $\sigma $, and $ w$, the inequality \eqref{e:g} holds if and only if these two conditions hold,
uniformly over all cubes $ I \subset \mathbb R ^{n}$,
\begin{align}\label{e:A2}
\frac {\sigma (I)} {\lvert  I\rvert } \cdot \frac {w (I)} {\lvert  I\rvert } &\le \mathscr A_2,
\\ \label{e:testing}
\iint _{Q (I)} (\nabla P_t (\sigma \mathbf 1_I)(x, t))^2 \; dw \,  t dt & \le \mathscr T ^2 \sigma (I) , \qquad
Q (I):= I \times [0, l (Q)].
\end{align}
Moreover, letting $ \mathscr A_2$ and $ \mathscr T$ denote the best constants above, there holds $ \mathcal G \simeq
\mathscr A_2 ^{1/2} + \mathscr T := \mathscr N$.
\end{theorem}

In general, there are substantive convergence issues in the two weight setting. Let us discuss them here.
We can assume that $ \sigma $ and $ w$ are restricted to some large cube.  For general $ f\in L ^2 (\sigma )$,
we can write it as the difference of its positive and negative parts, $ f= f_+ - f_-$.
Then, $ f _{\pm} \cdot \sigma $ is also a weight, i.e.\thinspace it is locally finite. Now, for each fixed $ t_0>0$, and all $t> t_0 $
the term $ \nabla P_ t (\sigma f _{\pm})$ is an integral against $ \sigma  \cdot f _{\pm}$, against a bounded (and signed) function,
on a compact set.  Hence it has an unambiguous definition. And we take $\nabla P_ t (\sigma f ) = \nabla P_ t (\sigma f _{+} )
- \nabla P_ t (\sigma f _{-}) $.  Thus, the integral $ \int _{\mathbb R ^{n}}\int _{t_0} ^{\infty }  |\nabla P_t f (x,t)|^2 t\; dt w (dx)$
has an unambigious definition.  Our estimates will be independent of the choice of $ t_0$, so we suppress it below.

The condition \eqref{e:A2} is the famous Muckenhoupt condition, in the two weight setting, and the condition \eqref{e:testing} is the
Sawyer testing condition, so named because of the foundational results of Eric Sawyer on the two weight theorems for the
maximal function \cite{Saw1} and the fractional integrals \cite{Saw2}.
Two weight inequalities for dyadic operators are relatively advanced, but their continuous versions are much harder.
Innovative work of Nazarov-Treil-Volberg on non-homogenous harmonic analysis \cites{NTV, NTV3,V}
renewed attention on these questions, and in particular the paper \cite{NTV3}, giving sufficient conditions for a two weight
inequality for singular integrals, has proven to be quite influential, in its use of the so-called \emph{pivotal condition.}

The critical observation in this paper is Lemma~\ref{l:pivotal}, which shows that the sufficient conditions \eqref{e:A2} and \eqref{e:testing} imply
a pivotal condition, like that introduced by Nazarov-Treil-Volberg \cite{NTV3}.
This is then combined with the averaging over good Whitney regions of  Martikainen-Mourgoglou \cite{MM},
followed by a  stopping time construction from \cite{LSSU1}. Also see \cite{LM} for the application of these two ideas in the local $ Tb$ setting.

\smallskip
In the setting of $ A_2$, and for $ A_p$ weights, the sharp estimates are due to Wittwer \cite{MR1897458}, and Lerner \cite{Ler3}, respectively.
The dyadic two weight theorem of the square function is a direct analog of our main theorem, and is easy to obtain. As far as we know, this is the
first general for a continuous two weight theorem for the square function known.

There was some speculation that the pivotal condition would be the main tool in the analysis of the Hilbert transform.
While not true, it turned out to be a crucial hint to determining that the correct condition is the so-called energy condition of \cite{LSU1}.
(This paper, by example, shows that the pivotal condition is not necessary from the boundedness of the Hilbert transform. This and much more is
	in \cite{Lac1}.)
The energy condition is crucial condition for the study of the two weight inequality for the Hilbert transform, as was shown in
\cites{LSSU1,Lac}.  On the other hand, the usefulness of the concept of energy for the study of other singular integrals is in doubt
due to the examples in \cite{SSU1}.

\smallskip

We first discuss the use of random dyadic grids, and martingale expansions, followed by  the necessity of the $ A_2$ condition, and then, in the sufficient direction, that the pivotal condition 
is also necessary from the testing and $ A_2$ condition. The next two sections concern the proof of sufficiency,
and the final section briefly outlines a corresponding result for the intrinsic square function.

\section{Random Grids, Martingales}

We recall the  function class $\mathcal{U}_{1,1}$,  of Wilson \cite[Page 114]{W}.
We call $\psi\in \mathcal{U}_{1,1}$ if $\psi$ satisfies the following
\begin{enumerate}
\item $\int_{\bbR^n} \psi(x)dx=0$;
\item $|\psi(x)|\le C(1+|x|)^{-n-1}$;
\item $|\psi(x)-\psi(x')|\le C|x-x'|((1+|x|)^{-n-2}+(1+|x'|)^{-n-2})$.
\end{enumerate}
It is easy to check that $\partial_t P _t \vert _{t=1} , \partial_{x_i} P_1 \in \mathcal{U}_{1,1}$, $i=1,\cdots,n$, and so to prove the main theorem, it
suffices to show the inequality below, where $ \psi _t (y) = t ^{-n} \psi (y/t)$,
\begin{equation}\label{e:suffices}
\int _{\mathbb R ^{n}}\int _{0} ^{\infty }  \lvert   \psi _t \ast (f \cdot \sigma )\rvert ^2 \frac {dt}t \, dw
\lesssim \mathscr N ^2 \lVert f\rVert_{\sigma } ^2 .
\end{equation}
Here, and throughout, we abbreviate $ \lVert f\rVert_{L ^2 (\sigma )}$ to $ \lVert f\rVert_{\sigma }$, since we are only interested in $ L ^2 $ norms.
\smallskip

We use random dyadic grids, the fundamental technique.  Let us denote the
random dyadic grid $\mathcal{D}=\mathcal{D}(\beta)$, where
$\beta=(\beta_j)_{j=-\infty}^\infty\in (\{0,1\}^n)^{\bbZ}$. That is,
\[
  \mathcal{D}=\Bigl\{Q+\sum_{j:2^{-j}<l(Q)}2^{-j}\beta_j: Q\in\mathcal{D}_0\Bigr\},
\]
 where
$\mathcal{D}_0$ is the standard dyadic grid of $\bbR^n$. Let $0<\gamma = \frac 1 {2(n+1)}<1$.
A cube $Q\in\mathcal{D}$ is called \emph{bad} if there exists another cube
$\tilde{Q}\in\mathcal{D}$ so that $l(\tilde{Q})\ge 2^r l(Q)$ and
$d(Q,\partial \tilde{Q})\le l(Q)^{\gamma}l(\tilde{Q})^{1-\gamma}$. Otherwise it
is \emph{good}. Note that $\pi_{\good}:=\mathbb{P}_\beta(\mbox{$Q+\beta$ is good})$
is independent of $Q\in\mathcal{D}_0$. The parameter $r$ is a fixed integer sufficiently large such that
$\pi_{\good}>0$.

Then, following \cite{MM},  we show that the integral in \eqref{e:suffices} equals
\begin{equation}\label{e:goodWR}
\int_0^\infty\int_{\bbR^n}|\psi_t*(f\sigma)|^2 \; w (dx) \frac{dt}{t}\\
=\frac{1}{\pi_{\good}}\mathbb{E}_\beta\sum_{R\in\mathcal{D}_{\good}}
\iint_{W_{R}}|\psi_t*(f\sigma)|^2 \; w(dx)\frac{dt}{t},
\end{equation}
where $W_R=R\times (l(R)/2, l(R)]$ is a Whitney region.
Then, our task is to show that for each $ \beta $, the sum on the right obeys the estimate in \eqref{e:suffices}.
In fact, with the monotone convergence theorem, it suffices to show that there exists a constant $C>0$ such that for any $s\in \mathbb N$,
we have
\begin{equation}\label{eq:monotone}
\sum_{R\in\mathcal{D}_{\good}\atop l(R)\le 2^s}
\iint_{W_{R}}|\psi_t*(f\sigma)|^2 \; w(dx)\frac{dt}{t}\le C \mathscr N^2 \|f\|_\sigma^2.
\end{equation}
Below, it is always understood that we only consider good cubes, and this is suppressed in the notation. The reduction to
\emph{good} Whitney regions is very useful, and we comment on it when it is used.

To see \eqref{e:goodWR}, write
\begin{align*}
\textup{LHS} \eqref{e:goodWR}
&=\mathbb{E}_\beta\sum_{R\in\mathcal{D}_0}\iint_{W_{R+\beta}} |\psi_t*(f\sigma)|^2 \; w(dx)\frac{dt}{t}
\\
&=\frac{1}{\pi_{\good}}\sum_{R\in\mathcal{D}_0}\mathbb{E}_\beta (1_{\good}(R+\beta))
\mathbb{E}_\beta \iint_{W_{R+\beta}}|\psi_t*(f\sigma)|^2 \; w(dx)\frac{dt}{t}
\\
&=\frac{1}{\pi_{\good}}\sum_{R\in\mathcal{D}_0}\mathbb{E}_\beta \left(1_{\good}(R+\beta)
\iint_{W_{R+\beta}}|\psi_t*(f\sigma)|^2\; w(dx)\frac{dt}{t}\right)
\\
&=\frac{1}{\pi_{\good}}\mathbb{E}_\beta\sum_{R\in\mathcal{D}_{\good}}
\iint_{W_{R}}|\psi_t*(f\sigma)|^2\;w(dx)\frac{dt}{t},
\end{align*}
where the independence of the position of $R+\beta$ and the goodness of $R+\beta$ is used (see \cite[Page 1479]{2912709}).

\medskip

Next we introduce the martingale decomposition. Define
\begin{gather*}
  E_Q^\sigma f:=\frac{1}{\sigma(Q)}\int_Q f\; d\sigma,
\end{gather*}
assuming that $ \sigma (Q)>0$, otherwise set it to be zero.
For the martingale differences,
\begin{gather*}
  \Delta_Q^\sigma f:=\sum_{\substack{Q'\in \textup{ch}(Q)}}  (E_{Q'}^\sigma f -E_Q^\sigma f)1_{Q'}.
\end{gather*}
(So if $ \sigma $ charges only one child of $ Q$, the martingale difference is identically zero.)
For fixed $s\in\mathbb N$,
 by Lebesgue differentiation theorem, we can write
\[
  f=\sum_{Q\in\mathcal{D}\atop l(Q)\le 2^s}\Delta_Q^\sigma f
  +\sum_{Q\in\mathcal{D}\atop l(Q)=2^s}(E_Q^\sigma f) 1_Q .
\]
By  orthogonality, we have
\[
  \|f\|_\sigma^2=\sum_{Q\in\mathcal{D}\atop l(Q)\le 2^s}\|\Delta_Q^\sigma f\|_\sigma^2
  +\sum_{Q\in\mathcal{D}\atop l(Q)= 2^s}\|(E_Q^\sigma f) 1_Q\|_\sigma^2.
\]

\smallskip

The final reduction of this section is  as follows. We can assume that $ f \in L ^2 (\sigma )$ is compactly supported.
With probability 1, the support of $ f$ is contained in a sequence of dyadic cubes in the random grid,
which increase to $ \mathbb R ^{n}$.  By monotone convergence, it suffices to consider a single dyadic cube $ Q ^0$,
of side length $ 2 ^{s}$,
 containing the support of $ f$, and bound the term below.
 \begin{equation*}
\iint_{Q ^{0} \times [0, 2 ^{s}]}\big| \psi_t*(f\sigma)\big|^2w(dx)\frac{dt}{t}
\lesssim \mathscr N ^2  \|f\|_\sigma^2.
\end{equation*}
Using the testing inequality, we can further assume that $ f$ has $ \sigma $-integral zero.
Then, $ f$ is in the linear span of the martingale differences, and we need only prove
\begin{align}\label{eq:e3}
 \sum_{R\in\mathcal{D}_{\good}\atop l(R)\le 2^s}
\iint_{W_{R}}\bigg|\sum_{Q\in\mathcal{D}\atop l(Q)\le 2^s}\psi_t*(\Delta_Q^\sigma  f \cdot \sigma)\bigg|^2w(dx)\frac{dt}{t}
\lesssim \mathscr N ^2  \|f\|_\sigma^2,
\end{align}
where  recall that $ W_R=R\times (l(R)/2, l(R)]$.

\section{Necessary Conditions}

First, we show that the Muckenhoupt $ A_2$ condition is necessary for the assumed norm inequality.

\begin{lemma}\label{l:A2}   The inequality \eqref{e:A2} holds, assuming the norm boundedness condition \eqref{e:g}.
\end{lemma}
\begin{proof}
Since
\[
  \partial_t P_t(x)=c_n\frac{|x|^2-nt^2}{(t^2+|x|^2)^{(n+3)/2}},
\]
we have
\begin{eqnarray*}
|\nabla u(x,t)|&\ge& |\partial_t P_t*f(x)|\\
&=&\left|\int_{\bbR^n}c_n\frac{|x-y|^2-nt^2}{(t^2+|x-y|^2)^{(n+3)/2}}f(y)dy\right|.
\end{eqnarray*}
For some fixed cube $Q$, let $f(y)=\mathbf{1}_Q(y)$.
Then for $x\in Q$,
\begin{eqnarray*}
g(f\sigma)(x)^2&\ge& \int_{2l(Q)}^{4l(Q)}\left|\int_{Q}c_n\frac{|x-y|^2-nt^2}{(t^2+|x-y|^2)^{(n+3)/2}}\sigma(dy)\right|^2 tdt
\gtrsim \frac{\sigma(Q)^2}{|Q|^2}.
\end{eqnarray*}
Now by the boundedness of $g(\cdot\sigma)$ from $L^2(\sigma)$ to $L^2(w)$, we get
\begin{eqnarray*}
\frac{\sigma(Q)^2}{|Q|^2} w(Q)
\lesssim \|g(f\sigma)\|_{L^2(w)}^2
\le\|g(\cdot\sigma)\|_{L^2(\sigma)\rightarrow L^2(w)}^2\sigma(Q),
\end{eqnarray*}
which is exactly the bound
$
  [w,\sigma]_{A_2}= \mathscr A_2 \lesssim  \mathcal  G ^2
$.

\end{proof}

\begin{Remark}
In fact, a stronger half-Poisson condition is necessary, namely
\begin{equation}\label{eq:e15}
\int_{\bbR^n}\frac{l(Q)^2}{(l(Q)+\dist(y, Q))^{2(n+1)}}\sigma(dy) \cdot w(Q)\lesssim \|g(\cdot\sigma)\|_{L^2(\sigma)\rightarrow L^2(w)}^2.
\end{equation}
Indeed,  let
\[
  f_0(y)=\frac{l(Q)}{(l(Q)+\dist(y, Q))^{n+1}}\mathbf 1_{\bbR^n\setminus 4\sqrt{n}Q}(y).
\]
Then for $x\in Q$ and $0<t<l(Q)$, by similar calculation, we can get
\[
  \int_{\bbR^n\setminus 4\sqrt{n}Q}\frac{l(Q)^2}{(l(Q)+\dist(y, Q))^{2(n+1)}}\sigma(dy)w(Q)\lesssim \|g(\cdot\sigma)\|_{L^2(\sigma)\rightarrow L^2(w)}^2.
\]
Combining this with the Muckenhoupt $A_2$ condition, we get the Poisson type condition.
\end{Remark}

Next we introduce the \emph{pivotal condition}.   We will need these preliminaries.  Below, we refer to dyadic grids, and the `good' integer $ r$,
and exponent $ \gamma $.

\begin{definition}\label{d:whitney} Given a dyadic cube $ I$, we set $ \mathcal W_I$ to be the maximal dyadic cubes $ K\subset I$
such that $ 2 ^{r} l (K) \le l (I)$ and $ \textup{dist} (K, \partial I) \ge  l (K) ^{\gamma } l (I) ^{1- \gamma }$.
\end{definition}

Above, $ \mathcal W$ stands for `Whitney'.  The next easy example shows that the collections $ \mathcal W_I$ `capture' good cubes which are strongly
contained in $ I$, and that the collections are Whitney, in that expanded by a fixed amount, they have bounded overlaps.
\begin{proposition}\label{p:whitney}
\begin{enumerate}
\item  For any good $ J \Subset I$,   there is a cube $ K\in \mathcal W_{I}$ which contains $ J$

\item  For any $ C>0$, provided $ r $ is sufficiently large, depending upon $ \gamma $,  there holds
\begin{equation*}
\sum_{K\in \mathcal W_I} \mathbf 1_{CK} \lesssim \mathbf 1_{I}
\end{equation*}
\end{enumerate}

\end{proposition}

Here, $ J\Subset I$ means that $ J\subset I$ and   $ 2 ^{r} l (J) \le l (I)$; in words, \emph{$ J$ is strongly contained in $ I$}.
For part 2, we only need a constant that is a function of dimension.

\begin{proof}
For part 1, observe that good intervals satisfy a stronger set of conditions than the intervals  $K\in  \mathcal W_I$.

\smallskip
For part 2,
without loss of generality, we can assume that $C\ge 3$. In this case, for any dyadic cube $K$, we have $F(K)\subset C K$, where $F(K)$ denotes the
dyadic father of $K$. Suppose that there are $ K_1 ,\dotsc, K_t \in \mathcal W_I$, with decreasing side lengths, such that
the intersection $ \bigcap _{s=1} ^{t} C K _{s} \neq 0$. We will show that $t< 2(1+\gamma^{-1})$ if $r$ is sufficiently large. In fact, if $t\ge2(1+\gamma^{-1})$, we have
\begin{eqnarray*}
\textup{dist}(F(K_t), \partial I)&\ge& \textup{dist}(C K_t, \partial I)\ge \textup{dist}(CK_1, \partial I)-C\sqrt n l(K_t)\\
                                 &\ge& \textup{dist}(K_1, \partial I)-C\sqrt n (l(K_t)+l(K_1))\\
                                 &\ge& l(K_1)^{\gamma}l(I)^{1-\gamma}-2C\sqrt n l(K_1)\\
                                 &\ge& (1-\frac{2C\sqrt n}{2^{r(1-\gamma)}})l(K_1)^{\gamma}l(I)^{1-\gamma}\\
                                 &\ge& (1-\frac{2C\sqrt n}{2^{r(1-\gamma)}})2l(F(K_t))^{\gamma}l(I)^{1-\gamma}.
\end{eqnarray*}
Let $r$ be sufficient large such that $2C\sqrt n\le 2^{r(1-\gamma)-1}$. Then
\[
  \textup{dist}(F(K_t), \partial I) \ge l(F(K_t))^{\gamma}l(I)^{1-\gamma},
\]
which is a contradiction of the definition of $K_t\in \mathcal W_I$.

\end{proof}

\begin{definition}\label{d:pivotal}  The \emph{pivotal constant $ \mathscr P$} is the smallest constant in the following inequality.
For any cube $ I^0$, and any partition of $ I^0$ into dyadic cubes $ \{I _{\alpha } \::\: \alpha \in \mathbb N \}$, there holds
\begin{equation}\label{e:pivotal}
\sum_{\alpha \in \mathbb N  } \sum_{K\in \mathcal W _{I_ \alpha }}
P(K, \mathbf 1_{I^0}\sigma)^2 w(K)\le \mathscr{P} ^2 \sigma(I ^{0}),
\end{equation}

\end{definition}

We show that the $ A_2$ and testing inequalities imply  the finiteness of the pivotal constant.

\begin{lemma}\label{l:pivotal} There holds $ \mathscr P  \lesssim  \mathscr A_2 ^{1/2} + \mathscr T := \mathscr N$.

\end{lemma}

\begin{proof}
We take the constant $ C$ in Proposition~\ref{p:whitney} to be $ C= 4 n$.  We first make a hole in the
argument of the Poisson term.
By the simple $ A_2$ bound,  there holds
\begin{align*}
\sum_{\alpha \in \mathbb N  } \sum_{K\in \mathcal W _{I_ \alpha }}
P(K, \mathbf 1_{ C K} \sigma)^2 w(K) \lesssim \mathscr A_2
\sum_{\alpha \in \mathbb N  } \sum_{K\in \mathcal W _{I_ \alpha }}  \sigma (C K) \lesssim \mathscr A_2 \sigma (I ^{0}).
\end{align*}
So, below, we can consider the Poisson terms   $ P(K, \mathbf 1_{  I ^{0} \setminus C K} \sigma) $.

Then, note that for $ x \in K$, and $ \tfrac 12 l (K) \le t \le l (K)$, there holds
\begin{equation*}
t \partial _t P _{t} ( \mathbf 1_{  I ^{0} \setminus C K} \sigma ) (x) \gtrsim P(K,\mathbf 1_{  I ^{0} \setminus C K} \sigma) .
\end{equation*}
And, so
\begin{equation*}
P(K, \mathbf 1_{ I ^{0} \setminus  C K} \sigma)^2 w(K) \lesssim \int _{K}  \int _{\tfrac 12 l (K)} ^{ l (K)}  t ^2  \lvert  \nabla  P _{t} ( \mathbf 1_{  I ^{0} \setminus C K} \sigma ) (x) \rvert ^2 \; d w  \, \frac {dt} t .
\end{equation*}
The sum over $ \alpha \in \mathbb N $ and $ K\in \mathcal W _{I_ \alpha }$ is split into two terms, splitting the argument of the
$ g$-function on the right.
The first is when the argument is $ \mathbf 1_{I ^{0}} \sigma $
\begin{align*}
\sum_{\alpha \in \mathbb N  } \sum_{K\in \mathcal W _{I_ \alpha }}
\int _{K}  \int _{\tfrac 12 l (K)} ^{ l (K)}  t ^2  \lvert  \nabla  P _{t} ( \mathbf 1_{  I ^{0} } \sigma ) (x) \rvert ^2 \; d w  \, \frac {dt} t
\lesssim \iint _{Q (I^0)} (\nabla P_t ( \mathbf 1_{I^0}\sigma)(x, t))^2 \; dw \, t dt \lesssim \mathscr T^2 \sigma (I ^{0}).
\end{align*}
The second is when the argument is $ \mathbf 1_{CK} \sigma $.  The Whitney property is again essential.
\begin{align*}
\sum_{\alpha \in \mathbb N  } \sum_{K\in \mathcal W _{I_ \alpha }}
\int _{K}  \int _{\tfrac 12 l (K)} ^{ l (K)}  t ^2  \lvert  \nabla  P _{t} ( \mathbf 1_{  CK } \sigma ) (x) \rvert ^2 \; d w  \, \frac {dt} t
&\lesssim
\sum_{\alpha \in \mathbb N  } \sum_{K\in \mathcal W _{I_ \alpha }}
\iint _{Q (CK)} (\nabla P_t (\mathbf 1_{CK}\sigma )(x, t))^2 \; dw \, t dt
\\&\lesssim \mathscr T ^2
\sum_{\alpha \in \mathbb N  } \sum_{K\in \mathcal W _{I_ \alpha }}   \sigma (C K) \lesssim \mathscr T ^2 \sigma (I ^{0}).
\end{align*}
The proof is complete.

\end{proof}

\section{The proof of Theorem~\ref{t:g}: The Core}
We   prove \eqref{eq:e3}, assuming that $ \mathscr N  := \mathscr A_2 ^{1/2}  + \mathscr T  $ is finite.
The core of the proof is the estimate for the following term, which we take up in this section.
\begin{equation}\label{e:critical}
  \sum_{R\;:\; l(R)\le 2^{s-r-1}}
\iint_{W_{R}}\bigg|\sum_{k=r+1}^{s-\log_2 l(R)}\psi_t*(\Delta_{R^{(k)}}^\sigma f \cdot  1_{R^{(k-1)}}\sigma)\bigg|^2w(dx)\frac{dt}{t}.
\end{equation} 
Recall that $ f$ is of $ \sigma $-integral zero on $ Q ^{0}$.    
Indeed, we fix an integer $ 0\le r' <  r+1$, and assume that
\begin{equation*}
f = \sum_{\substack{I\subset Q^0\\ \textup{$ I$ is good, }  }} \mathbf 1_{ \log_2 l (I) = r'  \operatorname{mod} r+1} \, \Delta ^{\sigma } _{I} f .
\end{equation*}
We can then further assume that $ \log_2 l (Q^0) = r' \mod r+1$, and that it is good.
Let $ \mathcal D _{f}$ be the dyadic children of good cubes $ I \subset Q^0$ with $  \log_2 l (I) = r' \mod r+1$.

\subsection{Stopping cubes}
The stopping cubes  $ \mathcal S$ are constructed this way.
Set $ \mathcal S_0$  to be all the maximal   dyadic children of $ Q^0$, which are 
 in $ \mathcal D_f$.
Then set $ \tau (S)= \mathbb E ^{\sigma  } _{S} f$, for $ S\in \mathcal S_0$.
In the recursive step,  assuming that $ \mathcal S_k$ is constructed, for $ S\in \mathcal S_k$,  set $ \textup{ch} _{\mathcal S} (S)$
to be the maximal subcubes $I\subset S$, $ I\in \mathcal D_f$,  such that either
\begin{enumerate}
\item   $ E ^{\sigma } _{I} \lvert  f\rvert > 10 \tau (S)$
\item   The first condition fails, and    $  \sum_{K\in \mathcal W_I} P (K, \mathbf 1_{S} \sigma ) ^2 w (K) \ge C_0 \mathscr P ^2 \sigma (I)$.
\end{enumerate}
Then, define $ \mathcal S _ {k+1} := \bigcup _{S\in \mathcal S_k} \textup{ch} _{\mathcal S} (S)$, and
\begin{equation*}
\tau (\dot S) :=
\begin{cases}
 E ^{\sigma } _{\dot S} \lvert  f\rvert  &  E ^{\sigma } _{\dot S} \lvert  f\rvert > 2 \tau (S)
 \\
 \tau (S) & \textup{otherwise}
\end{cases}, \qquad  \dot S\in \textup{ch} _{\mathcal S} (S).
\end{equation*}
Finally, $ \mathcal S := \bigcup _{k=0} ^{\infty } \mathcal S_k$.
It is a useful point that $ l (\dot S) \le 2 ^{-r-1} l (S)$ for all $ \dot S\in \textup{ch} _{\mathcal S} (S)$.
In particular, it follows that
\begin{equation} \label{e:byConstruct}
\dot S ^{(1)} \subset K, \quad \textup{for some $ K\in \mathcal W _{S}$}.
\end{equation}
This holds since $ \dot S ^{(1)} $ is good, and  strongly contained in $ S$,  so that Proposition~\ref{p:whitney} gives the implication above.

 For any dyadic cube $I$, $S(I)$ will denote its father in $ \mathcal S$, the minimal cube in $ \mathcal S$ that contains it. Notice that there maybe the case  $S(I)=I$.  For any stopping cube $S$, $\mathscr{F}(S)$ will denote its father in the stopping tree, inductively, $\mathscr F^{k+1} S=\mathscr{F}(\mathscr F^k S)$.

We collect simple consequences of the construction. See for instance \cite[Lemma 3.6]{LSSU1} for the easy proof.

\begin{lemma}\label{l:stop-simple} These estimates hold.
\begin{enumerate}
\item For all intervals $ I$,   $ \lvert  E ^{\sigma } _I f \rvert \lesssim \tau ( S (I)) $.


\item  This \emph{quasi-orthogonality} bound holds:
\begin{equation}\label{e:quasi}
  \sum_{S\in \mathcal S} \tau (S) ^2  \cdot \sigma (S) \lesssim \lVert f\rVert_{\sigma } ^2 .
\end{equation}

\end{enumerate}

\end{lemma}

With the tool of stopping cubes, we can make the following decomposition.
\begin{equation}\label{e:critical-all}
\begin{split}  &
\sum_{k=r+1}^{s-\log_2 l(R)} (E_{R^{(k-1)}}^\sigma\Delta_{R^{(k)}}^\sigma f) \mathbf{1}_{R^{(k-1)}}
\\&= \sum_{m=1}^\infty \sum_{k=r+1}^{s-\log_2 l(R)}    \mathbf 1_{  \mathscr F ^{m}S(R^{(r)}) \subset S (R ^{(k-1)})} (E ^{\sigma } _{R ^{ (k-1)}}\Delta_{R^{(k)}}^\sigma f)\mathbf{1}_{  \mathscr F ^{m}S(R^{(r)})\setminus \mathscr F ^{m-1} S(R^{(r)})   }
\\
&\quad+\sum_{k=r+1}^{s-\log_2 l(R)}(E_{R^{(k-1)}}^\sigma\Delta_{R^{(k)}}^\sigma f) \mathbf{1}_{S(R^{(r)})}
- \sum_{k=r+1}^{s-\log_2 l(R)}(E_{R^{(k-1)}}^\sigma\Delta_{R^{(k)}}^\sigma f) \mathbf{1}_{S(R^{(k-1)})\setminus R^{(k-1)}}.
\end{split}
\end{equation}

\subsection{The Paraproduct Estimate}
We consider the second term on the right in \eqref{e:critical-all}, in which we take the value of the martingale difference on $ R ^{ (k-1)}$, and pair that with the indicator of $\mathbf 1_{ S(R^{(r)})}$.
Then, for fixed $ S\in \mathcal S$,
\begin{align} \label{e:critical1}
 \sum_{\substack{R\;:\; l(R)\le 2^{s-r-1}\\ S(R^{(r)}) =S}}
\iint_{W_{R}} &\bigg|
\sum_{k=r+1}^{s-\log_2 l(R)}  \psi _t \ast  (E ^{\sigma } _{R ^{ (k-1)}}\Delta_{R^{(k)}}^\sigma f \cdot  \sigma\mathbf{1}_{   S } )\bigg|^2w(dx)\frac{dt}{t}
\\
& \lesssim \tau (S)  ^2
\iint _{Q (S)}  \bigl\lvert  \psi _t \ast (\sigma\mathbf 1_{S}) \bigr\rvert ^2 w (dx) \frac {dt} t \lesssim \mathscr N ^2 \tau (S) ^2 \sigma (S).
\end{align}
It is important to observe that the stopping value $ \tau (S)$ controls the sum over the martingale differences, permitting a simple
application of the testing condition.
And the sum over $ S$ is controlled by the quasi-orthogonality bound, \eqref{e:quasi}.

\subsection{The Global Bound}
We consider the first term on the right in \eqref{e:critical-all}, which concerns the case of
 case of $ S(R^{(r)})$ and $ S (R ^{(k)})$ being separated in the $ \mathcal S$ tree.  The stopping values enter in this way.
Fix a stopping cube $ S$ and integer $ m$. Note that for a choice of constant $ \lvert  c\rvert \lesssim 1 $, there holds
\begin{equation} \label{e:m-stop}
\sum_{k=r+1}^{s-\log_2 l(R)}    \mathbf 1_{  \mathscr F ^{m}S \subset S (R ^{(k-1)})} (E ^{\sigma } _{R ^{ (k-1)}}\Delta_{R^{(k)}}^\sigma f)
\mathbf 1_{  \mathscr F ^{m}S\setminus \mathscr F ^{m-1} S   }
= c \cdot  \tau ( \mathscr F ^{m}S) \mathbf 1_{  \mathscr F ^{m}S\setminus \mathscr F ^{m-1} S}.
\end{equation}
It is important to note the restriction on $ S (R ^{(k-1)})$ above.  We will gain a geometric decay in $ m$, from the pivotal condition.

We are going to reindex the sum above.
So, consider $ \ddot S \in \mathcal S$, and write integer $ m = p+q$, where $ p =\lceil m/2\rceil$.
Consider the sub-partition of $ \ddot S $ given by $ \mathcal P (m, \ddot S) = \{ \dot S \in \mathcal S \::\:
\mathscr F ^{p} \dot S = \ddot S  \}$.
Now, for stopping cube  $ S$ with $ \mathscr F ^{q}S=\dot S$, and good $ R\Subset S$, we
have $ R\subset \dot K$ for some $ \dot K\in \mathcal W_{\dot S}$, where $ \dot S \in \mathcal P$.
Notice that we have $ R\Subset \dot K\subset \dot S $. The hypothesis of
of Lemma~\ref{l:good-gain} holds for these three intervals, by goodness of $ R$.

There holds for each $ \ddot S \in \mathcal S$,
\begin{align*}
\sum_{R \::\: \mathscr F ^{m} S (R ^{(r)})   = \ddot S}&  \iint_{W_{R}} \Biggl|
\sum_{\substack{k=r+1  \\ \ddot S \subset \mathscr F ^{(m-1)} S (R ^{(k-1)})} }^{s-\log_2 l(R)}
E ^{\sigma } _{R ^{ (k-1)}}\Delta_{R^{(k)}  } ^{\sigma } f   \cdot   \psi _t \ast   \bigl( \sigma \cdot \mathbf 1_{  \ddot S\setminus \mathscr F ^{m-1} S   } \bigr) \Biggr|^2
\;w(dx)\frac{dt}{t}
\\& \stackrel{\eqref{e:m-stop}}\lesssim
\tau (\ddot S) ^2 \sum _{\dot S \in \mathcal P (m, \ddot S)}  \sum_{\dot K\in \mathcal W _{\dot S}}
\sum_{R \::\: R \subset \dot K}
\iint_{W_{R}} |
\psi _t \ast  \bigl\{ \sigma \cdot \mathbf{1}_{\ddot S \setminus \mathscr F ^{(p-1)} \dot S    } \bigr\} 	
|^2
\;w(dx)\frac{dt}{t}
\\
& \stackrel{\eqref{e:good-gain}}\lesssim
\tau (\ddot S) ^2 \sum _{\dot S \in \mathcal P (m, \ddot S)}
 \sum_{\dot K\in \mathcal W _{\dot S}}
P (\dot K, \sigma \mathbf 1_{\ddot S}) ^2
\sum_{R \::\: R  \subset \dot K}
\Bigl[ \frac { l (R)} {l (\dot K)} \Bigr]    w (R)
\\
& \lesssim  2 ^{- m/8}
\tau (\ddot S) ^2 \sum _{\dot S \in \mathcal P (m, \ddot S)}
 \sum_{\dot K\in \mathcal W _{\dot S}}
P (\dot K, \sigma \mathbf 1_{\ddot S}) ^2  w (\dot K)
\\&
\stackrel{\eqref{e:pivotal}}\lesssim \mathscr N ^2 \cdot  2 ^{-m/8}\tau (\ddot S) ^2  \sigma (\ddot S).
\end{align*}
We have applied \eqref{e:m-stop} to bound the sum over martingale differences,  and \eqref{e:good-gain}.
It is essential to note that we gain a geometric decay in $ m$, and that we have use the pivotal condition above.
The quasi-orthogonality bound \eqref{e:quasi} controls the sum over $ \ddot S \in \mathcal S$.

The gain of the geometric factor is explained this way. We can assume that $ q $ is greater than $ 2$.
Now, $ S (R)=S$ and $ \mathscr F ^{q} S=\dot S$.  Write the stopping cubes between $ S$ and $ \dot S$ as
\begin{equation*}
R\subset
S=S_1 \subsetneq S_2 \subsetneq \cdots \subsetneq S_q := \dot S, \qquad S _{t}\in \mathcal S,\ 1\le t \le q.
\end{equation*}
Recalling \eqref{e:byConstruct}, we see that $ S _{q-1} \subset \dot K$,
for $ \dot K\in \mathcal W _{\dot S}$ as above.  So, we have  $ l (R)\le 2 ^{-q+1} l (\dot K)$.  Since $ q \simeq m/2$, we have the geometric decay in $ m$ above.

\subsection{The Local Term}

We bound the third term on the right in \eqref{e:critical-all}.  The stopping rule on the pivotal condition is now essential.
Fix an $ S\in \mathcal S$, and fix a $ k \ge r$.  We will gain a geometric decay in $ k$, and sum in $ S$ with an orthogonality
bound, not quasi-orthogonality.

The central claim is the inequality below, in which we in addition fix a (good) cube  $ \dot R$ which intersects $ S$,
and child $ \ddot R$ of $ \dot R$.
\begin{align*}
 \lvert E_{\ddot R}^\sigma\Delta_{\dot R}^\sigma f \rvert ^2
\sum_{\substack{R \::\: S(\ddot R) = S\\ R ^{(k-1)} = \ddot R}} &
\iint _{W_R} \lvert  \psi _t \ast \mathbf{1}_{S\setminus \ddot R}   \rvert ^2 \; w (dx) \frac {dt} t
\lesssim 2 ^{- k/2} \mathscr N ^2  \lvert E_{\ddot R}^\sigma\Delta_{\dot R}^\sigma f \rvert ^2   \sigma (\ddot R ).
\end{align*}
It is clear that we can sum over the various fixed quantities to complete the proof in this case.

But the claim follows from construction of the stopping interval. Since $ S (\ddot R) =S$, this means that the cube $ \ddot R$
must fail the conditions of the stopping cube construction, in particular it must fail the pivotal stopping condition.  Hence, using \eqref{e:good-gain},
for each $ K\in \mathcal W _{\ddot R}$,
\begin{align*}
\sum_{\substack{R \::\: S(\ddot R) = S\\ R ^{(k-1)} = \ddot R \,,\ R\subset K}} &
\iint _{W_R} \lvert  \psi _t \ast \mathbf{1}_{S\setminus \ddot R}   \rvert ^2 \; w (dx) \frac {dt} t
\lesssim 2 ^{-k/2} P (K, \sigma \mathbf 1_{S}) ^2 w (K)
\end{align*}
And, the sum over $ K\in \mathcal W _{\ddot R} $ of this last expression can't be more than $ \mathscr N ^2 \sigma (\ddot R)$,
since $ \ddot R$ is not a stopping cube.  This completes this case.

\bigskip
The use of the argument above goes back to the paper \cite{NTV3}.
We have appealed to the Lemma below, which is one of the principal consequences of the good rectangles.  In the two weight setting,
it also appears in \cite{NTV3}, and has been used in (seemingly) every subsequent paper. (Interestingly, its role in the characterization of the
	two weight Hilbert transform \cites{LSSU1,Lac} is not nearly so central as it is in this argument.)

\begin{lemma}\label{l:good-gain}  Consider three cubes  $ R  \subset K \subset  S$, and function $ f $ not supported on $ S$.
If $ \textup{dist} (R, \partial K) \ge l (R) ^{\gamma } { l (K)} ^{1- \gamma }  $, then   there holds
\begin{equation}\label{e:good-gain}
\iint _{W_R} \lvert  \psi _t \ast (f \cdot \sigma )\rvert ^2 \; w (dx) \frac {dt} t
\lesssim \frac { l (R)} {l (K)}  \cdot 
P (K, \lvert  f\rvert \sigma  ) ^2  w (R).
\end{equation}
\end{lemma}

\begin{proof}
For $ (x,t) \in W _R$, and $ y \not\in S$,
\begin{align*}
\lvert  \psi _t (x-y) \rvert &\lesssim \frac { l (R)} { l (R) ^{n+1} + \textup{dist} (y, R) ^{n+1} }
\\
& \lesssim  \frac {  l (R) } { l (K)} \cdot
\frac { l (K)} { l (R)  ^{n+1}+ \textup{dist} (y, R) ^{n+1} }
\\
& \lesssim
\biggl[ \frac {  l (R) } { l (K)}  \biggr] ^{1- (n+1) \gamma }
\frac { l (K)} {  l(K)^{n+1} + \textup{dist} (y, K) ^{n+1} }
\end{align*}
Recall that $ \gamma =\frac{1}{2(n+1)}$, then integrating this inequality against $ \lvert  f\rvert \cdot \sigma  $ will prove the inequality.

\end{proof}

\section{The Elementary Estimates}
\subsection{Some Lemmas}
 First of all, we study the two weight inequality of the averaging operators  defined as follows.
\[
  A^\sigma_rf(x):=\frac{1}{r^n}\int_{Q(x,r)}f(y)\sigma(dy),
\]
where $Q(x,r)$ is the cube with center $x$ and side-length $r$.
We have the following well-known result.
\begin{Lemma}\label{lm:l1}
Suppose that $(w,\sigma)$ satisfies the Muckenhoupt $A_2$ condition.
Then
\[
  \|A^\sigma_rf\|_{L^2(w)}\le C_n\mathscr A_2 ^{1/2}\|f\|_\sigma.
\]
\end{Lemma}

There is a straight-forward domination of Poisson averages as sums of simple averages.
\begin{gather*}
  K_{\alpha,t}(x,y):=\frac{t^\alpha}{(t+|x-y|)^{n+\alpha}},\quad \alpha, t>0, x,y\in \bbR^n,
  \\
  \textup{and} \qquad
  I_{\alpha, t}^\sigma f (x):=\int_{\bbR^n}K_{\alpha,t}(x,y) f(y)\sigma(dy).
\end{gather*}
There holds
\begin{eqnarray*}
I_{\alpha, t}^\sigma |f| (x)&\le& \sum_{k=0}^\infty\int_{2^{k}<|x-y|/t\le 2^{k+1}}t^{-n} 2^{k(-n-\alpha)}|f(y)|\sigma(dy)\\
&&\quad+ \int_{|x-y|/t\le 1}t^{-n} |f(y)|\sigma(dy)\\
& \lesssim & \sum_{k=-1}^\infty 2^{-k\alpha}A_{2^{k+2} t}^\sigma |f| (x).
\end{eqnarray*}
It follows from Lemma~\ref{lm:l1} that
\begin{equation}\label{eq:e9}
\|I_{\alpha, t}^\sigma f\|_{L^2(w)}\le C_{n,\alpha}\mathscr A_2 ^{1/2}\|f\|_\sigma.
\end{equation}

This Lemma summarizes some standard off-diagonal considerations.
\begin{Lemma}\label{lm:l2}
Let
\[
  A_{QR}^\alpha:=\frac{l(Q)^{\alpha/2}l(R)^{\alpha/2}}{D(Q,R)^{n+\alpha}}\sigma(Q)^{1/2}w(R)^{1/2},
\]
where $D(Q,R)=l(Q)+l(R)+d(Q,R)$, $Q,R\in\mathcal{D}$ and $\alpha>0$. Then
for $x_Q, y_R\ge 0$, we have the following estimate on a bilinear form.
\[
  \sum_{Q,R\in\mathcal{D}}A_{QR}^\alpha x_Qy_R\lesssim \mathscr A_2 ^{1/2}\bigg(\sum_{Q\in\mathcal{D}}x_Q^2
  \times
 \sum_{R\in\mathcal{D}}y_R^2\bigg)^{1/2}.
\]
\end{Lemma}

\begin{proof}
We bound `half' of the form, the bound for the complementary half following by duality.  The half we consider is
\begin{align*}
\sum_{R\in\mathcal{D}}\sum_{Q\in\mathcal{D}\atop l(Q)< l(R)} &
\frac{l(Q)^{\alpha/2}l(R)^{\alpha/2}}{D(Q,R)^{n+\alpha}}\sigma(Q)^{1/2}w(R)^{1/2}x_Q y_R\\
&=\sum_{j=1}^\infty 2^{-j\alpha/2}\sum_{k\in\bbZ}\sum_{l(R)=2^k\atop l(Q)=2^{k-j}}\frac{2^{k\alpha}}{D(Q,R)^{n+\alpha}}\sigma(Q)^{1/2}w(R)^{1/2}x_Q y_R\\
&:=\sum_{j=1}^\infty 2^{-j\alpha/2}\sum_{k\in\bbZ}S_{j,k}.
\end{align*}
It suffices to obtain a uniform bound for the terms $ \sum_{k\in\bbZ} S _{j,k}$.

Set
\[
  h(x)=\sum_{Q\in\mathcal{D}\atop l(Q)=2^{k-j}}\sigma(Q)^{-1/2}x_Q 1_Q(x),\quad \tilde{h}(y)=\sum_{R\in\mathcal{D}\atop l(R)=2^k}w(R)^{-1/2}y_R 1_R(y).
\]
It is obvious that
\begin{align*}
S_{j,k}&\le\iint K_{\alpha,2^{k}}(y,x)h(x) \tilde{h}(y)\sigma(dx) w(dy)\\
&= \int_{\bbR^n}I_{\alpha, 2^k}^\sigma (h)(y) \tilde{h}(y)w (dy)\\
& \lesssim  \mathscr A_2 ^{1/2}\|h\|_\sigma\cdot\|\tilde h\|_{L^2(w)}\quad & \mbox{(by \eqref{eq:e9})}\\
&= \mathscr A_2 ^{1/2}\bigg(\sum_{Q\in\mathcal{D}, l(Q)=2^{k-n}}x_Q^2
\times \sum_{R\in\mathcal{D}, l(R)=2^k}y_R^2\bigg)^{1/2}.
\end{align*}
This is a bound for $ S _{j,k}$, and a uniform bound on $ \sum_{k\in\bbZ} S _{j,k}$ follows by Cauchy--Schwarz.  This completes the proof.

\end{proof}

We return to the proof of \eqref{eq:e3}, collecting estimates that are complementary to the core estimate \eqref{e:critical}.
Firstly, by the size condition of the kernel, for $(x,t)\in W_R$, we have
\begin{equation}\label{eq:size}
| \psi_t*(\Delta_Q^\sigma f \cdot \sigma)(x)|\lesssim \frac{l(R)}{(l(R)+d(Q,R))^{n+1}}\sigma(Q)^{1/2}\|\Delta_Q^\sigma f\|_\sigma.
\end{equation}
Secondly, for $Q$ with $l(Q)<2^{s}$, by the cancellation condition $\int_Q \Delta_Q^\sigma f \; \sigma (dx)=0$, we can write
\[
  \psi_t*(\Delta_Q^\sigma f \cdot \sigma)(x)=\int_Q (\psi_t(x-y)-\psi_t(x-\zeta_Q))
  (\Delta_Q^\sigma f)(y)\sigma(dy),
\]
where $\zeta_Q$ is the center of $Q$. Since $|y-\zeta_Q|\le \sqrt n l(Q)/2
\le \sqrt n l(R)/4<\sqrt n t/2\le\sqrt n l(R)/2$, we have
\begin{eqnarray*}
|\psi_t(x-y)-\psi_t(x-\zeta_Q)|&\le&
\frac{Ct^{-n}|\frac{y-\zeta_Q}{t}|}{(1+|\frac{x-y}{t}|)^{n+2}+(1+|\frac{x-\zeta_Q}{t}|)^{n+2}}\\
&=&\frac{C|y-\zeta_Q|t}{(t+|x-y|)^{n+2}+(t+|x-\zeta_Q|)^{n+2}}\\
&\lesssim&\frac{l(Q)l(R)}{(l(R)+d(Q,R))^{n+2}}.
\end{eqnarray*}
Therefore,
\begin{equation}\label{eq:derivative}
| \psi_t*(\Delta_Q^\sigma f \cdot \sigma)(x)|\lesssim \frac{l(Q)l(R)}{(l(R)+d(Q,R))^{n+2}}\sigma(Q)^{1/2}\|\Delta_Q^\sigma f\|_\sigma,\,\, (x,t)\in W_R.
\end{equation}
\subsection{The case $l(Q)<l(R)$}
In this subsection, we estimate the following.
\[
  \sum_{R\in\mathcal D_\good\atop l(R)\le 2^s}
\iint_{W_{R}}\bigg|\sum_{Q\in\mathcal{D}\atop l(Q)<l(R)}\psi_t*(\Delta_Q^\sigma f \cdot \sigma)\bigg|^2w(dx)\frac{dt}{t}.
\]
Notice that in this case, we must have $l(Q)<2^s$, by \eqref{eq:derivative} and Lemma~\ref{lm:l2},
\begin{align*}
\sum_{R\in\mathcal D_\good\atop l(R)\le 2^s}
\iint_{W_{R}} & \bigg|\sum_{Q\in\mathcal{D}\atop l(Q)<l(R)}\psi_t*(\Delta_Q^\sigma f \cdot \sigma)\bigg|^2w(dx)\frac{dt}{t}\\
&\lesssim\sum_R \bigg(\sum_Q A_{QR}^1 \|\Delta_Q^\sigma f\|_\sigma\bigg)^2 \lesssim  \mathscr A_2   \|f\|_\sigma^2.
\end{align*}

\subsection{The case $l(Q)\ge l(R)$ and $d(Q,R)>l(R)^\gamma l(Q)^{1-\gamma}$}
In this subsection, we estimate the following
\[
   \sum_{R\;:\; l(R)\le 2^s}
\iint_{W_{R}}\bigg|\sum_{Q\in\mathcal{D},l(Q)\ge l(R)\atop d(Q,R)>l(R)^\gamma l(Q)^{1-\gamma}}\psi_t*(\Delta_Q^\sigma  f \cdot \sigma)\bigg|^2\; w(dx)\frac{dt}{t}.
\]
If $l(Q)\le d(Q,R)$, it is trivial that
\[
  \frac{l(R)}{(l(R)+d(Q,R))^{n+1}}\le \frac{l(R)}{D(Q,R)^{n+1}}.
\]
If $l(Q)>d(Q,R)$, then $D(Q,R)\lesssim l(Q)$. Since $d(Q,R)>l(R)^\gamma l(Q)^{1-\gamma}$, we have
\begin{eqnarray*}
\frac{l(R)}{(l(R)+d(Q,R))^{n+1}}&\le& \frac{l(R)}{l(R)^{(n+1)\gamma}l(Q)^{(n+1)(1-\gamma)}}.
\end{eqnarray*}
Set $\gamma=\frac{1}{2(n+1)}$, we have
\begin{eqnarray*}
\frac{l(R)}{(l(R)+d(Q,R))^{n+1}}&\lesssim&
\frac{l(R)^{1/2}}{D(Q,R)^{n+1/2}}.
\end{eqnarray*}
Now it follows from \eqref{eq:size} and Lemma~\ref{lm:l2} that
\begin{align*}
\sum_{R\;:\; l(R)\le 2^s}
\iint_{W_{R}} &\bigg|\sum_{Q\in\mathcal{D},l(Q)\ge l(R)\atop d(Q,R)>l(R)^\gamma l(Q)^{1-\gamma}}\psi_t*(\Delta_Q^\sigma f \cdot \sigma)\bigg|^2w(dx)\frac{dt}{t}\\
&\lesssim \sum_R \bigg(\sum_Q A_{QR}^{1/4}\|\Delta_Q^\sigma f\|_\sigma\bigg)^2
\lesssim \mathscr A_2 \|f\|_\sigma^2.
\end{align*}

\subsection{The case $l(R)\le l(Q)\le 2^r l(R)$ and $d(Q,R)\le l(R)^\gamma l(Q)^{1-\gamma}$}
In this case, $D(Q,R)\sim l(Q) \sim l(R)$.  Then by \eqref{eq:size} and Lemma~\ref{lm:l2} again,
\begin{eqnarray*}
\sum_{R\;:\; l(R)\le 2^s}
\iint_{W_{R}}\bigg|\sum_{l(R)\le l(Q)\le 2^r l(R)\atop d(Q,R)\le l(R)^\gamma l(Q)^{1-\gamma}}\psi_t*(\Delta_Q^\sigma  f \cdot \sigma)\bigg|^2w(dx)\frac{dt}{t}
\lesssim \mathscr  A_2 \|f\|_\sigma^2.
\end{eqnarray*}

\subsection{The case $l(Q)> 2^r l(R)$ and $d(Q,R)\le l(R)^\gamma l(Q)^{1-\gamma}$}
In this case, since $R$ is good, we must have $R\subset Q$.
But, recalling \eqref{e:critical-all}, we have already accounted for the term $ (\Delta_{R^{(k)}}^\sigma f) \mathbf 1_{R^{(k-1)}}$.
Accordingly, here we need only consider the term $ \mathbf 1_{R^{(k)}\setminus R^{(k-1)}} \Delta_{R^{(k)}}^\sigma f$.

The estimate \eqref{e:good-gain} applies as below, in which we hold the summing variable $ k\ge r+1$ constant.
\begin{align*}
\sum_{R \::\:  l(R)\le 2^{s-r-1}} &
\iint_{W_{R}}\big| \psi_t*((\mathbf 1_{R^{(k)}\setminus R^{(k-1)}}\Delta_{R^{(k)}}^\sigma f)\sigma)\big|^2w(dx)\frac{dt}{t}\\
&\lesssim2^{-k}  \sum_{R \::\:  l(R)\le 2^{s-r-1} } P (R ^{(k)}, \lvert \Delta_{R^{(k)}}^\sigma f \rvert \cdot \sigma  ) ^2 \cdot {w(R)}
\\
&\lesssim 2^{-k}
 \sum_{I}  \lVert \Delta ^{\sigma } _{I} f\rVert_{\sigma } ^2 \frac {\sigma (I)} {\lvert  I \rvert } \cdot \frac {w (I)} {\lvert  I\rvert }
\lesssim  2^{-k}  \mathscr A_2 \|f\|_\sigma^2,
\end{align*}
where   we reindexed the sum over $ R$ above.
With the geometric decay in $ k$, this estimate clearly completes the proof.

\section{The Intrinsic Square Function}\label{sec:s1}
In this section, we will expand our previous results to the intrinsic square functions.
For $0<\alpha\le 1$, let $\mathcal{C}_\alpha$ be the family of functions supported in $\{x: |x|\le 1\}$, satisfying $\int \varphi=0$, and such that for all $x$ and $x'$,
$|\varphi(x)-\varphi(x')|\le |x-x'|^\alpha$. If $f\in L_{\rm{loc}}^1(\bbR^n)$ and
$(y,t)\in \bbR_+^{n+1}$, we define
\[
  A_\alpha(f)(y,t)=\sup_{\varphi\in\mathcal{C}_\alpha}|f*\varphi_t(y)|.
\]
Then the intrinsic square function is defined by
\[
  G_{\beta,\alpha}(f)(x)=\bigg(\int_{\Gamma_\beta(x)}A_\alpha(f)(y,t)^2\frac{dy dt}{t^{n+1}}\bigg)^{1/2},
\]
where $\Gamma_\beta(x)=\{(y,t): |y-x|<\beta t\}$. If $\beta=1$, set $G_{1,\alpha}(f)=G_\alpha(f)$.
The intrinsic $g$-function is defined by
\[
  g_\alpha(f)(x)=\bigg(\int_{0}^\infty A_\alpha(f)(x,t)^2\frac{ dt}{t}\bigg)^{1/2}.
\]
We mention some properties on $G_{\beta,\alpha}$ and $g_{\alpha}$.  And we
refer the readers to \cite{W} for details.
\begin{enumerate}
\item $G_{\beta,\alpha}(f)\sim G_\alpha(f) \sim g_\alpha(f)$.
\item The classic $g$-function satisfies $g(f)\le c G_\alpha(f)$.
\end{enumerate}
Now we can state our result on the intrinsic $g$-function, which is equivalent
to the intrinsic square function by the above properties.
\begin{Theorem}
Suppose that $w$ and $\sigma$ are positive Borel measures on $\bbR^n$. Then $g_{\alpha}(\cdot\sigma)$ is bounded from $L^2(\sigma)$ to $L^2(w)$ provided that
\begin{enumerate}
\item $(w,\sigma)$ satisfies the Muckenhoupt $A_2$ condition;
\item the testing condition holds, i.e.
\[
  \mathcal{T}^2=\sup_{R: \mbox{cubes in $\bbR^n$}}\frac{1}{\sigma(R)}\iint_{Q(R)} A_\alpha(\sigma 1_R)(x,t)^2 w(dx)\frac{dt}{t}<\infty.
\]
\end{enumerate}
Moreover,
$
  \|g_{\alpha}(\cdot\sigma)\|_{L^2(\sigma)\rightarrow L^2(w)}
    \lesssim \mathscr A_2 ^{1/2} +\mathcal{T}
$.

\end{Theorem}

Since the intrinsic square function dominates the classical one, the boundedness of the intrinsic square function
implies the Muckenhoupt $ A_2$ condition. This also implies that the testing condition and $ A_2$ condition
imply the pivotal condition. The remaining details of the proof of sufficiency are much like the proof given,
so we omit the easy details.

\end{document}